\documentclass[12pt]{article}
\usepackage[utf8]{inputenc}
\usepackage{geometry}
\usepackage{float}
\usepackage{latexsym,bm}
\usepackage{amsmath,amsfonts,amsthm,amssymb,bbding,mathrsfs,mathtools}
\usepackage{fancyhdr}
\usepackage{graphicx}
\usepackage{hyperref}
\usepackage{xcolor}
\usepackage{authblk}
\usepackage{newunicodechar}
\newunicodechar{，}{,}

\newtheorem{thm}{Theorem}[section]

\newtheorem{lemma}{Lemma}[section]
\newtheorem{cor}{Corollary}[section]

\newtheorem{conj}{Conjecture}[section]
\newtheorem{claim}{Claim}[section]

\title{Pancyclicity in graph families with the Ore-type condition}

\author[1]{Luyi Li}
\author[1,2]{Yubo Wang}
\author[1,2,$\ast$]{Guiying Yan}

\affil[1]{\small Academy of Mathematics and Systems Science, Chinese Academy of Sciences, Beijing, China.}
\affil[2]{\small University of Chinese Academy of Sciences, Beijing, China.}
\affil[$\ast$]{\small Corresponding author: \texttt{yangy@amt.ac.cn}}

\date{}

\begin{document}
\maketitle

\begin{abstract}
Let $ n \in \mathbb{N} $ with $ n \geq 3 $, and let $\mathcal{G} = \{G_i:i\in [n]\} $ be a family of $ n $-vertex graphs on a common vertex set $V$, where the graphs in the family do not need to be distinct. A graph  $H$ with vertex set $V$ is  \emph{rainbow} in $\mathcal{G}$ if there exists an injection $ \phi: E(H) \to [n] $ such that $e \in E(G_{\phi(e)})$ for every edge $e \in E(H)$, where $|E(H)|\leq n$.  In 2020, Joos and Kim proved that $\mathcal{G}$ contains a rainbow Hamiltonian cycle under the Dirac-type condition. Recently, Liu, Chen, and Ma generalized this result by replacing the Dirac-type condition with a more general Ore-type condition involving degree sums of non-adjacent vertices: If $\sigma(\mathcal{G}) \geq n$, then $\mathcal{G}$ contains a rainbow Hamiltonian cycle, where the Ore-type condition $\sigma(\mathcal{G})$ is defined as follows:
$
\sigma(\mathcal{G}) = \min\{d_p(u) + d_q(v) \mid uv \notin E(G_i) \text{ for some } i \in [n] \text{ and for all } p, q \in [n]\}.
$
In this paper, under the Ore-type condition, we show that either each vertex of $V$ is contained in a rainbow cycle of length $\ell$ for every $\ell\in[4,n]$, or $G_1=\cdots=G_n=K_{\frac{n}{2},\frac{n}{2}}$. As a corollary, we deduce the rainbow pancyclicity of $\mathcal{G}$, which
 supports the famous meta-conjecture posed by Bondy. Furthermore, we prove rainbow vertex-pancyclicity of $\mathcal{G}$ under the Ore-type condition and provide an extremal graph family to show that the result is sharp.

\end{abstract}

\maketitle

\section{Introduction} 
   The study of Hamiltonian cycles and pancyclicity in graphs has been a central theme in graph theory for decades. Classical results such as Dirac's Theorem \cite{dirac1952} and Ore's Theorem \cite{ore1960} establish fundamental connections between degree conditions and the existence of Hamiltonian cycles. Dirac's Theorem states that every $n$-vertex graph $G$ with $\delta(G)\geq n/2$ contains a Hamiltonian cycle, where $\delta(G)$ denotes the minimum value of the degrees of all vertices in $G$. On the one hand, this result was subsequently strengthened by Ore: every $n$-vertex graph $G$ with $\sigma(G)\geq n$ contains a Hamiltonian cycle, where $\sigma(G)$ denotes the minimum value of the sums of the degrees among all pairs of non-adjacent vertices in $G$. Hence, $\delta(G)$ and $\sigma(G)$ are also referred to as Dirac-type and Ore-type conditions, respectively. There are many sufficient conditions for Hamiltonicity; see, for example, \cite{fan1984,ghouila1960}.
   On the other hand, under the same conditions in Dirac's Theorem, Bondy \cite{bondy1971pancyclic} proved a stronger result: every $n$-vertex graph $G$ with $\delta(G)\geq n/2$ is pancyclic or $G=K_{\frac{n}{2},\frac{n}{2}}$, where a graph is \emph{pancyclic} if it contains cycles of all possible lengths.  Then, Bondy \cite{Bondy-pancyclic} proposed the famous meta-conjecture: Almost any nontrivial sufficient condition for the Hamiltonicity of graphs can also guarantee the pancyclicity of graphs, except for maybe a simple family of exceptional graphs.
   To support this conjecture, Bondy \cite{bondy1971pancyclic} proved the following result:
   \begin{thm}[\cite{bondy1971pancyclic}]\label{Ore-pancyclic}
       Every $n$-vertex graph $G$ with $\sigma(G)\geq n$ is pancyclic or $G=K_{\frac{n}{2},\frac{n}{2}}$.
   \end{thm}
In recent years, there has been growing interest in extending these classical results to the setting of graph families, leading to the development of transversal graph theory. This area concerns the existence of subgraphs that intersect each graph in a family in a prescribed manner, often requiring that different edges come from different graphs in the families. 
Let $n,a ,b$ be positive integers with $a \leq b $. We write $[n]=\{1,2,\cdots,n\}$ and $[a,b]=\{a,a+1,\cdots, b\}$; in particular, $[a,a]=\{a\}$ and $[0]=\varnothing $.
Let $n\in \mathbb{N}$ and $n\geq 3$, suppose $\mathcal{G}=\{G_{i}: i\in [t]\}$ is a family of $n$-vertex graphs with the same vertex set $V$, and the graphs in the family can be identical. For a graph $H$ with vertex set $V$ and $|E(H)|\leq t$, we say that $\mathcal{G}$ contains a \emph{rainbow} subgraph $H$ if there exists a mapping $\phi:E(H)\rightarrow[t]$ such that for each edge $e\in E(H)$, we have $e\in E(G_{\phi(e)})$; and if $|E(H)|=t$, i.e., $\phi$ is a bijection, then we say that $\mathcal{G}$ contains a \emph{transversal} isomorphic to $H$. A family $\mathcal{G}=\{G_{i}: i\in [n]\}$ of $n$-vertex graphs is \emph{rainbow pancyclic} if $\mathcal{G}$ contains a rainbow cycle of length $\ell$ for each $\ell\in[3,n]$. A family $\mathcal{G}=\{G_{i}: i\in [n]\}$ of $n$-vertex graphs is \emph{$[a,b]$-rainbow vertex-pancyclic} if each vertex of $V$ is contained in a rainbow cycle of length $\ell$  for each $\ell\in[a,b]$.
In particular, the family $\mathcal{G}$ is \emph{rainbow vertex-pancyclic} if each vertex of $V$ is contained in a rainbow cycle of length $\ell$ for each $\ell\in[3,n]$. If a rainbow path connecting $x$ and $y$ exists, let $d_G(x,y)$ denote the length of a shortest rainbow path connecting $x$ and $y$. The family $\mathcal{G}$ is \emph{rainbow panconnected} if, for any two vertices $x,y\in V$, there is a rainbow path of length $\ell$ starting at $x$ and ending at $y$ for each $\ell\in[d_{\mathcal{G}}(x,y),n-1]$.


In 2020, Aharoni et al.~\cite{aharoni2020rainbow} established a rainbow analogue of Mantel's theorem by determining an edge-density condition that guarantees a rainbow triangle in a graph family. Specifically, they proved that if $G_1,G_2,G_3$ are graphs on a common set of $n$ vertices and
\[
|E(G_i)| > \frac{1+\tau^2}{4}n^2
= \frac{26-2\sqrt{7}}{81}n^2
\quad \text{for each } i\in [3],
\]
where $\tau=\frac{4-\sqrt{7}}{9}$, then there exists a rainbow triangle in $\{G_1,G_2,G_3\}$. This work helped initiate the study of transversal and rainbow extremal problems for graph families. This result extends classical Mantel's theorem to graph families, establishing the edge number threshold for rainbow triangles. Furthermore, they proposed a conjecture regarding a transversal analogue of Dirac's theorem. This conjecture was solved by Cheng, Wang, and Zhao \cite{cheng2021rainbow} approximately. Eventually, it was resolved completely by Joos and Kim \cite{joos2020}. In their article, they stated:
\begin{thm}[\cite{joos2020}]\label{Joos-Hamiltonian}
Let $n\in \mathbb{N}$ and $n\geq 3$. Suppose $\mathcal{G}=\{G_{i}: i\in [n]\}$ is a family of $n$-vertex graphs with the same vertex set, and the graphs in the family can be identical. If for each $i\in[n]$, $\delta(G_{i})\geq n/2$, then there exists a transversal isomorphic to a Hamiltonian cycle, that is, $\mathcal{G}$ contains a rainbow Hamiltonian cycle.
\end{thm}
This achievement extends Dirac's theorem to the framework of transversals in graph families, laying the foundation for subsequent research. Just as Bondy generalized Dirac's theorem, Li et al. \cite{li2024} further considered the pancyclicity of a  family of graphs under the same conditions of Theorem \ref{Joos-Hamiltonian}, obtaining the following result:
\begin{thm}[\cite{li2024}]\label{minimum-degree}
Suppose $\mathcal{G}=\{G_{i}: i\in [n]\}$ is a family of $n$-vertex graphs with the same vertex set $V$, and the graphs in the family can be identical. If $\delta(G_{i})\geq n/2$ for each $i\in [n]$, then $\mathcal{G}$ is rainbow pancyclic or  $G_1=\cdots=G_n=K_{n/2,n/2}$.
\end{thm}     

\begin{thm}[\cite{li2023rainbow}]\label{panconnected}
Suppose $\mathcal{G}=\{G_{i}: i\in [n]\}$ is a family of $n$-vertex graphs with the same vertex set $V$, and the graphs in the collection can be identical. If $\delta(G_{i})\geq n/2+1$ for each $i\in [n]$, then $\mathcal{G}$ is rainbow panconnected.
\end{thm} 

It is natural to consider Ore's theorem in a family of graphs. 
Bradshaw \cite{bradshaw2021transversals} and Li, Li and Li \cite{li2023rainbow} proposed that it is also interesting to consider those conditions in a family of graphs.
\begin{conj}\label{Conj-Ore}
Let $n\in \mathbb{N}$ and $n\geq 3$. Suppose $\mathcal{G}=\{G_{i}: i\in [n]\}$ is a family of $n$-vertex graphs with the same vertex set, and the graphs in the family can be identical. If for each $i\in[n]$, $\sigma(G_{i})\geq n$, then there exists a rainbow Hamiltonian cycle in $\mathcal{G}$.
\end{conj}
This conjecture has not yet been solved. 
Recently, Liu, Chen, and Ma \cite{liu2025hamiltonian} defined the following stronger Ore-type condition for a graph family $\mathcal{G} = \{G_i:i\in[t]\}$ :
\[
\sigma(\mathcal{G}) = \min\{d_p(u) + d_q(v) \mid uv \notin E(G_i) \text{ for some } i \in [t] \text{ and for all } p,q \in [t]\}.
\]
Under this Ore-type condition, the authors \cite{liu2025hamiltonian} solved Conjecture \ref{Conj-Ore}:
\begin{thm}[\cite{liu2025hamiltonian}]\label{Hamiltonian-cycle}
Let $n\in \mathbb{N}$ and $n\geq 3$. Suppose $\mathcal{G}=\{G_{i}: i\in [n]\}$ is a family of $n$-vertex graphs with the same vertex set, and the graphs in the family can be identical.  If $\sigma(\mathcal{G}) \geq n$, then there exists a rainbow Hamiltonian cycle in $\mathcal{G}$.
\end{thm}
Theorem \ref{Hamiltonian-cycle} is also a generalization of Ore's theorem in graph families. It gives a partial affirmative answer to Conjecture \ref{Conj-Ore}. For more results about this topic, please refer to \cite{aharoni2019large, chakraborti2024, cheng2023rainbow, cheng2025, ferber2022dirac, montgomery2020transversal}.
Inspired by the meta-conjecture and Theorem \ref{minimum-degree}, we continue to consider the pancyclicity of the family of graphs under the Ore-type condition. In this paper, we prove the following result:
\begin{thm}\label{[4,n]-vertex-pancyclicity}
    Let $\mathcal{G}=\{G_1,G_2,...,G_n\}$ be a family of $n$-vertex graphs on the same vertex set $V$. If $\sigma(\mathcal{G})\geq n$, then $\mathcal{G}$ is $[4,n]$-rainbow vertex-pancyclic or $G_1=\cdots=G_n=K_{\frac{n}{2},\frac{n}{2}}$.
\end{thm}
From Theorem \ref{[4,n]-vertex-pancyclicity}, we obtain the following corollary, which supports the meta-conjecture.

\begin{cor}\label{pancyclicity-our}
Let $n\in \mathbb{N}$ and $n\geq 3$. Suppose $\mathcal{G}=\{G_{i}: i\in [n]\}$ is a family of $n$-vertex graphs with the same vertex set, and the graphs in the family can be identical. If $\sigma(\mathcal{G})\geq n$, then $\mathcal{G}$ is rainbow pancyclic or $G_1=\cdots=G_n=K_{\frac{n}{2},\frac{n}{2}}$.
\end{cor}
\begin{proof}
By Theorem \ref{[4,n]-vertex-pancyclicity}, it suffices to show that $\mathcal{G}$ contains a rainbow triangle or $G_1=\cdots=G_n=K_{\frac{n}{2},\frac{n}{2}}$. If every edge in $\mathcal{G}$ is a strong edge, then the result follows by Theorem \ref{Ore-pancyclic}. Otherwise, there is an edge $xy$ in $\mathcal{G}$ such that $xy\notin E(G_a)$ for some $a\in [n]$. Without loss of generality, assume that $xy\in E(G_b)$. Then $d_a(x)+d_c(y)\geq n$ for some $c \in [n]\setminus\{a,b\}$. It is easy to show that $N_a(x)\cap N_c(y)\neq\emptyset$. Choose a vertex $z\in N_a(x)\cap N_c(y)$, then $xyzx$ is a rainbow triangle in $\mathcal{G}$. The result follows. 
\end{proof}

In \cite{li2024}, the authors also proved the rainbow vertex-pancyclicity of $\mathcal{G}=\{G_{i}: i\in [n]\}$ when $\delta(G_{i})\geq \frac{n+1}{2}$ for each $i\in [n]$. 
\begin{thm}[\cite{li2024}]\label{vertex-pancyclicity-minimum-degree}
Suppose $\mathcal{G}=\{G_{i}: i\in [n]\}$ is a family of $n$-vertex graphs with the same vertex set $V$, and the graphs in the family can be identical. If $\delta(G_{i})\geq \frac{n+1}{2}$ for each $i\in [n]$, then $\mathcal{G}$ is rainbow vertex-pancyclic.
\end{thm}  

\begin{figure}
    \centering
    \includegraphics[width=0.5\linewidth, trim={0.5cm 1.2cm 0.5cm 0.6cm},clip]{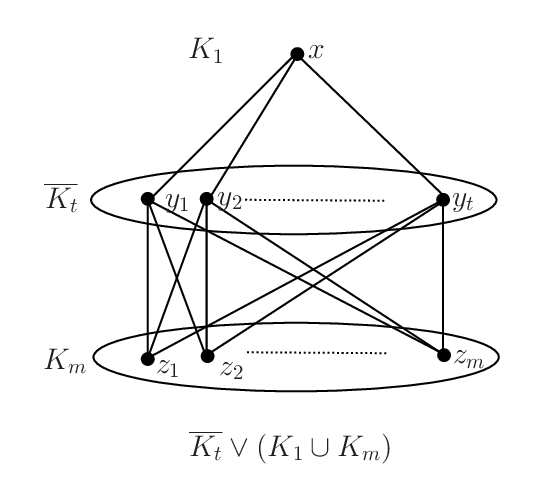}
    \caption{The extremal graph $\overline{K_t} \vee (K_1 \cup K_m)$}
    \label{fig}
\end{figure}

Similarly, we want to generalize Theorem \ref{vertex-pancyclicity-minimum-degree} when  $\sigma(\mathcal{G})\geq n+1$. The complement of $G$, denoted by $\overline{G}$, 
is the graph on the same vertex set $V$ in which two distinct vertices 
are adjacent if and only if they are non-adjacent in $G$.
However, let  $\mathcal{G}=\{G_1,G_2,...,G_n\}$ and $G_1 = \cdots = G_n = \overline{K_t} \vee (K_1 \cup K_m)$, where $t = \frac{n}{3} + \frac{2}{3}$, $m = \frac{2n}{3} - \frac{5}{3}$ and $n \equiv 1 \pmod{3}$. As shown in Figure \ref{fig}. It is not difficult to find that $x$ is not contained in a rainbow triangle in $\mathcal{G}$. Then the condition  $\sigma(\mathcal{G})\geq \frac{4n-4}{3}$  does not guarantee the rainbow vertex-pancyclicity of $\mathcal{G}$. Hence, we prove the following result:
\begin{thm}\label{vertex-pancyclicity-ore}
    Let $\mathcal{G}=\{G_i: i\in [n]\}$ be a family of $n$-vertex graphs on the same vertex set $V$. If $\sigma(\mathcal{G})\geq \frac{4n}{3}-1$, then $\mathcal{G}$ is rainbow vertex-pancyclic.
\end{thm}


Here, we provide some notation and terminology.
In this paper, we follow standard graph-theoretic notation and terminology.  
A \emph{graph} $G = (V, E)$ consists of a finite set $V$ of \emph{vertices} and a set $E \subseteq \{uv \mid u, v \in V, u \neq v\}$ of \emph{edges}. All graphs considered are simple (without loops or multiple edges) and finite.
For a graph $G = (V, E)$ and a vertex $v \in V$, we use $N_G(v)$ to denote the \emph{neighborhood} of $v$ and $d_G(v)=|N_G(v)|$ to denote the \emph{degree} of $v$ in $G$. 
The \emph{minimum degree} of $G$ is defined $\delta(G)=\mbox{min}\{d_G(u): u\in V(G)\}$. 
Let $\mathcal{G} = \{G_i : i \in [n]\}$ be a family of $n$-vertex graphs with the same vertex set, and the graphs in the family can be identical.  For any two vertices $u$ and $v$,  
an edge $e=xy$ is called a \emph{strong edge} in $\mathcal{G}$ if $xy\in E(G_i)$ for each integer $i\in [n]$. A set of integers $c(e)$ is defined as $c(e) := \{ i \in [n] : e \in E(G_i) \}$.  Let $C_n$ be a cycle of length $n$ with vertices $x_1,x_2,\dots,x_n$ in cyclic order. 
We write $\overrightarrow{C_n}$ to denote the cycle traversed in the clockwise direction such as $x_1x_2\overrightarrow{C_n}x_nx_1$, and $\overleftarrow{C_n}$ to denote the same cycle traversed in the counterclockwise direction such as $x_1x_nx_{n-1}\overleftarrow{C_n}x_1$.
For basic concepts not defined here, we refer the reader to \cite{graph-theory}.



\section{Main Lemma}
To prove Theorem \ref{[4,n]-vertex-pancyclicity}, we first show that the following lemma on the existence of rainbow cycle of $n-1$ in the graph family $\mathcal{G}=\{G_1,G_2,...,G_n\}$.
\begin{lemma}\label{[n-1]-vertex-pancyclicity}
    Let $\mathcal{G}=\{G_1,G_2,...,G_n\}$ be a family of $n$-vertex graphs on the same vertex set $V$ with $\sigma(\mathcal{G})\geq n$. Then each vertex is contained in a rainbow cycle of length $n-1$ in $\mathcal{G}$ or $G_1=\cdots =G_n=K_{\frac{n}{2}, \frac{n}{2}}$.
\end{lemma}
\begin{proof}
By Theorem \ref{Hamiltonian-cycle}, let $C_n = x_1 x_2 \cdots x_n x_1$ be a rainbow Hamiltonian cycle with $x_i x_{i+1} \in E(G_i)$ for each $i\in [n]$, where the subscripts are taken modulo $n$. Without loss of generality, suppose to the contrary that there is no rainbow cycle of length $n-1$ containing $x_1$. Then $x_i x_{i+2} \notin E(G_i)\cup E(G_{i+1})$ for each $i\in [n-1]$.
Now we prove this result through the following claims:
\begin{claim}\label{lem-claim-1}
Each edge of $C_n$ is a strong edge.
\end{claim}
\begin{proof}
First, we prove that $x_{i-1}x_{i}\in E(G_{i})$ and $x_{i}x_{i+1}\in E(G_{i-1})$ for each $ i \in [2, n] $.
Since $ x_{i-1}x_{i+1} \notin E(G_{i-1}) \cup E(G_{i}) $, we have
$d_{i}(x_{i-1}) + d_{i-1}(x_{i+1}) \geq n$.
Let
$$A_{i} = \{a \in [n] \setminus \{i-1,i,i+1\} : x_{i-1}x_{a} \in E(G_{i})\}$$ and $$
A_{i-1} = \{a \in [n] \setminus \{i-2,i-1,i\} : x_{i+1}x_{a+1} \in E(G_{i-1})\}.
$$
Note that $|A_{i}| \geq d_{i}(x_{i-1}) - 1 $ and $ |A_{i-1}| \geq d_{i-1}(x_{i+1}) - 1 $. If $ A_{i} \cap A_{i-1} \neq \emptyset $, then there exists $ a \in A_{i} \cap A_{i-1} $, such that 
$x_{i-1}x_{a}\overleftarrow{C_{n}}x_{i+1}x_{a+1}\overrightarrow{C_{n}}x_{i-1}$
is a rainbow cycle of length $n-1$ containing $x_1$, a contradiction. Hence $ A_{i} \cap A_{i-1} = \emptyset $.
Since $A_{i}\cup A_{i-1}\subseteq[n]\setminus\{i-1,i\}$, we have $|A_{i}|+|A_{i-1}|=|A_{i}\cup A_{i-1}|\leq n-2$.
Combining with 
$$|A_{i}|+|A_{i-1}| \geq d_{i}(x_{i-1}) - 1+d_{i-1}(x_{i+1}) - 1\geq n-2,$$
it follows that $|A_{i}|+|A_{i-1}|=|A_{i}\cup A_{i-1}|=n-2$. 
By the definitions of $A_i$ and $A_{i-1}$, we have  $|A_{i}| = d_{i}(x_{i-1}) - 1 $ and $ |A_{i-1}| = d_{i-1}(x_{i+1}) - 1 $, which implies that $x_{i-1}x_i\in E(G_i)$ and $x_ix_{i+1}\in E(G_{i-1})$ for each $ i \in [2, n] $.

Note that $x_1x_2\in E(G_1)$ and $x_1x_2\in E(G_2)$. For each pair of integers $1 \leq a\leq j \leq n$, let $C_a(j)=x_1x_2\cdots x_nx_1$ be a rainbow Hamiltonian cycle such that
\[
\begin{cases}
    x_ix_{i+1}\in E(G_i), & \text{if } i\in [a-1]\cup [j+1,n]; \\
    x_ix_{i+1}\in E(G_j), & \text{if } i=a; \\
    x_ix_{i+1}\in E(G_{i-1}), & \text{if } i\in [a+1,j].
\end{cases}
\]
Note that $C_a(a)=C_n$ for any $a \in [1,n]$. Recall that $x_{i-1}x_{i}\in E(G_{i})$ and $x_{i}x_{i+1}\in E(G_{i-1})$ for each $ i \in [2, n] $. Then we can obtain the rainbow Hamiltonian cycle $C_{a}(j)$ from $C_{a+1}(j)$ for each $a\in[j-1]$.
It is clear that $C_i(j)$ is a rainbow Hamiltonian cycle with $x_ix_{i+1}\in E(G_j)$ for $j \in [i,n]$. \\
By a similar argument,
let $C^a(j)=x_1x_2\cdots x_nx_1$ be a rainbow Hamiltonian cycle such that
\[
\begin{cases}
    x_ix_{i+1}\in E(G_i), & \text{if } i\in [a-1]\cup [j+1,n]; \\
    x_ix_{i+1}\in E(G_a), & \text{if } i=j; \\
    x_ix_{i+1}\in E(G_{i+1}), & \text{if } i\in [a,j-1].
\end{cases}
\]
Note that $C^a(a)=C_n$ for any $a \in [1,n]$. Recall that $x_{i-1}x_{i}\in E(G_{i})$ and $x_{i}x_{i+1}\in E(G_{i-1})$ for each $ i \in [2, n] $. Then we can obtain the rainbow Hamiltonian cycle $C^{a}(j+1)$ from $C^a(j)$ for each $j\in[a,n-1]$. It is clear that $C^i(j)$ is a rainbow Hamiltonian cycle with $x_ix_{i+1}\in E(G_j)$ for $j \in [i]$. 

Therefore, $x_ix_{i+1}\in E(G_a)$ for each $a \in [n]$. Then we can obtain that each edge of $C_n$ is a strong edge, the claim follows.
\end{proof}

\begin{claim}\label{lem-claim-2}
   For any two distinct integers $i, j\in [n]$, we have $x_ix_{i+2}\notin E(G_j)$.
\end{claim}
\begin{proof}
If $x_ix_{i+2}\in E(G_j)$ for some integer $i \in [n-1]$ and some $j\in [n]$, then $x_ix_{i+2}\overrightarrow{C_n}x_i$ is a rainbow cycle of length $n-1$ containing $x_1$, where $x_ix_{i+2}\in E(G_j)$ and $x_jx_{j+1}\in E(G_i)$ by Claim \ref{lem-claim-1}. This contradicts the hypothesis that there is no rainbow cycle of length $n-1$ containing $x_1$. Then $x_ix_{i+2}\notin E(G_j)$ for each integer $i \in [n-1]$ and each integer $j\in [n]$.

If $x_2x_n\in E(G_a)$ for some integer $a\in [n-3]$, then $P=x_1 x_n x_2\overrightarrow{C_n}x_{n-2}$ is a rainbow path of length $n-2$ with $x_2x_n\in E(G_a)$ and    $x_ax_{a+1}\in E(G_1)$ by Claim \ref{lem-claim-1}. We know that $P$ does not use edges of $G_{n-2}$ or $G_{n-1}$.
If $x_1x_{n-2}\in E(G_{n-1})\cup E(G_{n-2})$, then $x_1x_nx_2 \overrightarrow{C_n} x_{n-2}x_1$ forms a rainbow cycle of length $n-1$ containing $x_1$, where $x_2x_n\in E(G_a)$, $x_ax_{a+1}\in E(G_1)$ and $x_1x_{n-2}\in E(G_{n-1})\cup E(G_{n-2})$.
This contradicts the hypothesis that there is no rainbow cycle of length $n-1$ containing $x_1$.
Next, we assume that $x_1x_{n-2}\notin E(G_{n-1})\cup E(G_{n-2})$. Then
$d_{n-1}(x_1) + d_{n-2}(x_{n-2}) \geq n$. 
For the sake of convenience, we relabel the rainbow path $P=x_1'x_2'\cdots x_{n-1}'$.
Define the sets:
$$
S = \{ i \in [3,n-2]: x'_{n-1}x'_i \in E(G_{n-2}) \} \mbox{ and }
T = \{ i \in [n-3] : x'_1 x'_{i+1} \in E(G_{n-1}) \}.
$$
It follows from $x_{n-2} x_n \notin E(G_{n-2})$ and $x_1x_{n-2}\notin E(G_{n-1})\cup E(G_{n-2})$ that $x'_{n-1}x_2'\notin E(G_{n-2})$ and $x'_1x'_{n-1}\notin E(G_{n-1})\cup E(G_{n-2})$. Then $|S| = d_{n-2}(x_{n-2}) - 1$. 
It follows from $x_1x_{n-2}, x_1x_{n-1}\notin E(G_{n-1})$ that $x'_1x_{n-1}', x'_1x_{n-1}\notin E(G_{n-1})$. Then $|T| = d_{n-1}(x_{1}).$ 
Therefore，
$|S| + |T| =  d_{n-2}(x_{n-2}) - 1+ d_{n-1}(x_{1})\geq n-1$.
However, $S\cup T\subseteq [n-2]$ and $|S \cup T|\leq n-2$, this implies that $S\cap T\neq \emptyset$. Choose an integer $k \in S \cap T$, then $x_1'x_{k+1}'\overrightarrow{P}x_{n-1}'x_k'\overleftarrow{P}x_1'$ is a rainbow cycle of length $n-1$ containing $x_1$, as shown in Figure \ref{fig-1}, where $x_1'x_{k+1}'\in E(G_{n-1})$ and $x_{n-1}'x_k'\in E(G_{n-2})$ by the definitions of $S$ and $T$, a contradiction.
Then we conclude that $x_2x_n\notin E(G_i)$ for each integer $i\in [n-3]$. 
\begin{figure}
    \centering
    \includegraphics[width=0.5\linewidth]{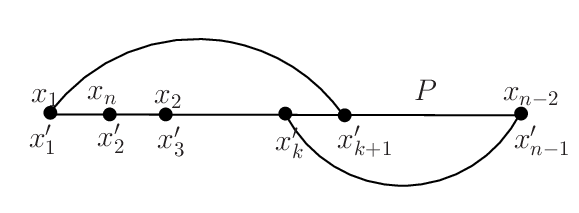}
    \caption{The rainbow cycle of length $n-1$}
    \label{fig-1}
\end{figure}

If $x_2x_n\in E(G_a)$ for some integer $a\in \{n-2,n-1,n\}$, then $P'=x_1 x_2 x_n\overleftarrow{C_n}x_4$ is a rainbow path of length $n-2$ with $x_2x_n\in E(G_a)$ and $x_ax_{a+1}\in E(G_n)$, by a similar discussion, we conclude that $x_2x_n\notin E(G_i)$ for each integer $i\in \{n-2,n-1,n\}$. Then $x_2x_n\notin E(G_i)$ for each integer $i\in [n]$.
The claim follows.
\end{proof}

\begin{claim}\label{lem-claim-3}
For any three integers $i,s,t \in [n]$, there is a rainbow cycle $C_{st}(i)$ of length $n-2$ such that $V=V(C_{st}(i)) \cup \{x_{i+1}, x_{i+2}\}$ and $C_{st}(i)$ contains no edges of $G_s \cup G_t$.
\end{claim}
\begin{proof}
By symmetry, we only need to prove the case of $i=1$. 
By Claim \ref{lem-claim-1}, $x_4x_5\overrightarrow{C_n}x_1$ is a rainbow path of order $n-2$ such that all edges are strong edges. 
Fix two integers $s\in [n]$ and $t\in [n]$, if there are two integers $a,b\in[n]\setminus\{s,t\}$ such that $x_1x_4\in E(G_a) \cup E(G_b)$, then 
$C_{st}(1)=x_4x_5\overrightarrow{C_n}x_1x_4$ is the desired cycle. The claim follows.

Then we assume that $x_1x_4\notin E(G_a) \cup E(G_b) $ for two distinct integers $a,b \in[n]\setminus\{s,t\}$. Then $d_a(x_1) + d_b(x_4) \geq n $.
Define the sets:
$$
J_a = \{ j \in [5, n] : x_1 x_j \in E(G_a) \}  \mbox{ and } 
J_b = \{ j \in [4,n-1] : x_{4} x_{j+1} \in E(G_{b}) \}.
$$
By Claim \ref{lem-claim-1} and Claim \ref{lem-claim-2}, since $x_1x_2$ and $x_3x_4$ are strong edges, and $x_1 x_{3} \notin E(G_a), x_2 x_{4} \notin E(G_b)$, then $|J_a|=d_a(x_1) - 1$ and $|J_b|= d_{b}(x_{4}) - 1.$
Therefore,
$$
|J_a| + |J_b| = d_a(x_1) - 1+ d_b(x_{4}) - 1 \geq n-2.
$$
However, $J_a \cup J_b\subseteq [4,n]$ and $|J_a \cup J_b|\leq n-3$, this implies that $J_a \cap J_b \neq \emptyset$. Choose an integer $k \in J_a \cap J_b$, then $C_{st}(1)=x_1 x_k \overleftarrow{C_n} x_{4} x_{k+1} \overrightarrow{C_n}x_1$ forms a rainbow  cycle of length $n-2$ such that $V=V(C_{st}(1)) \cup \{x_2, x_{3}\}$ where $x_1x_k \in E(G_a)$ and $x_4x_{k+1} \in E(G_b)$ and $x_ix_{i+1} \in \bigcup_{i \in [n] \setminus \{a,b,s,t\} } E(G_i)$ for $i \in [4,k-1] \cup [k+1,n]$. See Figure \ref{fig-2}. Then $C_{st}(1)$ contains no edges of $G_s \cup G_t$. The claim follows.
\end{proof}

\begin{figure}
    \centering
    \includegraphics[width=0.3\linewidth]{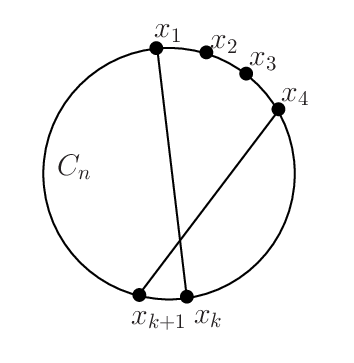}
    \caption{The rainbow cycle $C_{st}(1)$ of length $n-2$}
    \label{fig-2}
\end{figure}

\begin{claim}\label{lem-claim-4}
For all distinct $a,b \in [n]$ and $i \in [n]$,     we have $d_a(x_i)+d_b(x_i) \geq n$. 
\end{claim}
\begin{proof}
Assume, to the contrary, that $d_a(x_i)+d_b(x_i) < n$ for some distinct $a,b \in [n]$ and $i \in [n]$. We first prove that $d_a(x_i)+d_b(x_i) \geq n$ for all $i \in [n-3] \cup \{n-1,n\}$. Fix such an $i$, and suppose that $d_a(x_i)+d_b(x_i) < n$.
By Claim \ref{lem-claim-2}, we have $x_{i}x_{i+2}\notin E(G_j)$ for any integer $j \in [n]$ and $x_{n+1}=x_1$, $x_{n+2}=x_2$. Then 
$d_a(x_{i}) + d_a(x_{i+2}) \geq n $ and $d_b(x_{i}) + d_b(x_{i+2}) \geq n$ for any two distinct integers $a,b \in [n]$.
It follows that  $d_a(x_{i+2}) + d_b(x_{i+2}) > n$.
 By Claim \ref{lem-claim-3}
 let $C_{ab}(i+1)=y_1y_2 \cdots y_{n-2}y_1$ be a rainbow cycle of length $n-2$ such that $V=V(C_{ab}(i+1)) \cup \{x_{i+2}, x_{i+3}\}$ and $C_{ab}(i+1)$ contains no edges of $G_a \cup G_b$.
Hence, $C_{ab}(i+1)$ contains no edges of $G_a$ or $G_b$ and $d_a(x_{i+2}) + d_b(x_{i+2}) > n$ for two distinct integers $a,b \in [n]$.

 Now set $y_{n-1}=y_1=x_1$ and we define the following two sets:
    $$M=\{j\in[n-2]:x_{i+2}y_{j+1}\in E(G_a)\} \mbox{ and }N=\{j\in[n-2]:x_{i+2}y_{j}\in E(G_b)\}.$$
Note that $d_a(x_{i+2},V(C_{ab}(i+1)))+d_b(x_{i+2}, V(C_{ab}(i+1))) > d_a(x_{i+2})+d_b(x_{i+2})-2\geq n-2$. It follows from the definitions of $M$ and $N$ that $|M|+|N|> n-2$. Note that $M\cup N\subseteq [n-2]$, we have $M\cap  N \neq \varnothing $ ,so there is an integer $k \in M\cap N$, then the $y_{k}x_{i+2}y_{k+1}\overrightarrow{C_{ab}(i+1)}y_{k}$ is a rainbow cycle of length $n-1$ containing $x_1$, a contradiction. Then we have  $d_a(x_i)+d_b(x_i) \geq n$ for all $i \in [n-3] \cup \{n-1,n \}$. For $i=n-2$, assume to the contrary that
$d_a(x_{n-2})+d_b(x_{n-2})<n$. 
By Claim \ref{lem-claim-2}, we have $x_{i}x_{i+2}\notin E(G_j)$ for any integer $j \in [n]$ and $x_{n+1}=x_1$, $x_{n+2}=x_2$. Then 
$d_a(x_{n-2}) + d_a(x_{n}) \geq n $ and $d_b(x_{n-2}) + d_b(x_{n}) \geq n$ for any two distinct integers $a,b \in [n]$.
It follows that  $d_a(x_{n}) + d_b(x_{n}) > n$. By Claim \ref{lem-claim-3}
 let $C_{ab}(n-2)=y_1y_2 \cdots y_{n-2}y_1$ be a rainbow cycle of length $n-2$, hence by a  similar discussion we know that  $d_a(x_{n-2})+d_b(x_{n-2}) \geq n$. Then for all distinct $a,b \in [n]$ and $i \in [n]$, we have $d_a(x_i)+d_b(x_i) \geq n$.
 The claim follows.
\end{proof}
\begin{claim}\label{lem-claim-5}
For any $a,i \in [n]$, $N_a(x_i)=\{x_1,x_3,...,x_{n-3}, x_{n-1}\}$ when $i$ is even and $N_a(x_i)=\{x_2,x_4,...,x_{n-2}, x_n\}$ when $i$ is odd.
\end{claim}
\begin{proof}
First, we prove for any $i\in [n-1]$ and $a \in [n]$, $N_a(x_i)=\{x_1,x_3,...,x_{n-3}, x_{n-1}\}$ when $i$ is even and $N_a(x_i)=\{x_2,x_4,...,x_{n-2}, x_n\}$ when $i$ is odd.
   For convenience, we extend the notation by setting $x_0=x_n$ and $C_{ab}(0)=C_{ab}(n)$. 
  By Claim \ref{lem-claim-3}
  let $C_{ab}(i-1)=y_1y_2 \cdots y_{n-2}y_1$ be a rainbow cycle of length $n-2$ such that $V=V(C_{ab}(i-1)) \cup \{x_{i}, x_{i+1}\}$ and $C_{ab}(i-1)$ contains no edges of $G_a \cup G_b$.
So by Claims \ref{lem-claim-3} and \ref{lem-claim-4} we can have that $C_{ab}(i-1)$ contains no edges of $G_a$ or $G_b$ and $d_a(x_{i}) + d_b(x_{i})\geq n$ for any distinct integers $a,b \in [n]$  .
Now set $y_{n-1}=y_1=x_{i-1}$ and we define the following two sets:
    $$M=\{j\in[n-2]:x_{i}y_{j+1}\in E(G_a)\} \mbox{ and }N=\{j\in[n-2]:x_{i}y_{j}\in E(G_b)\}.$$
Note that $d_a(x_{i},V(C_{ab}(i-1)))+d_b(x_{i}, V(C_{ab}(i-1))) \geq d_a(x_{i})+d_b(x_{i})-2\geq n-2$. It follows from the definitions of $M$ and $N$ that $|M|+|N| \geq n-2$. Note that $M\cup N\subseteq [n-2]$. If there is an integer $k \in M\cap N$,  $y_{k}x_{i}y_{k+1}\overrightarrow{C_{ab}(i-1)}y_{k}$ is a rainbow cycle of length $n-1$ containing $x_1$, a contradiction. Then $M \cap N =\varnothing $ and  $M\cup N=[n-2]$. 
 Let $\{t,t+1,\cdots,t+k\}\subseteq M$ be a maximum subset such that $t-1\notin M$ and $t+k+1\notin M$. Then $t-1\in N$ and $t+k+1\in N$. If $k\geq 1$, then $x_iy_{t+k+1}\in E(G_b)$ and $x_{i}y_{t+k}\in E(G_a)$. It follows that $y_{t+k}x_iy_{t+k+1}\overrightarrow{C_{ab}(i-1)}y_{t+k}$ is a rainbow cycle of length $n-1$ containing $x_1$, a contradiction. Then $k=0$. Hence the sets $M$ and $N$ alternate on $[n-2]$, and therefore $|M|=|N|=\frac{n-2}{2}$.
Since $x_{i}x_{i+1}$ is a strong edge by Claim \ref{lem-claim-1}, we have $d_a(x_i)=|M|+1=\frac{n}{2}
\quad\text{and}\quad
d_b(x_i)=|N|+1=\frac{n}{2}$. 
Because the choice of the distinct colors $a,b\in [n]$ was arbitrary, it follows that $d_p(x_i)=\frac{n}{2}$ for every $p\in [n]$. 
Furthermore, we can know that 
$$N_a(x_i)=\{y_1,y_3,...,y_{n-3}, x_{i+1}\} \mbox{ or } N_a(x_i)=\{y_2,y_4,...,y_{n-2}, x_{i+1}\}$$and 
$$N_b(x_i)=\{y_1,y_3,...,y_{n-3}, x_{i+1}\} \mbox{ or } N_b(x_i)=\{y_2,y_4,...,y_{n-2}, x_{i+1}\}.$$

If $N_a(x_i)=\{y_1,y_3,...,y_{n-3}, x_{i+1}\}$ and $N_b(x_i)=\{y_2,y_4,...,y_{n-2}, x_{i+1}\}$, then $x_iy_2\overrightarrow{C_{ab}(i-1)}y_1x_i$ is a rainbow cycle of length $n-1$ containing $x_1$, a contradiction. By a similar argument, we can get a contradiction when $N_b(x_i)=\{y_1,y_3,...,y_{n-3}, x_{i+1}\}$ and $N_a(x_i)=\{y_2,y_4,...,y_{n-2}, x_{i+1}\}$. 
Then $$N_a(x_i)=N_b(x_i)=\{y_1,y_3,...,y_{n-3}, x_{i+1}\} \mbox{ or } N_a(x_i)=N_b(x_i)=\{y_2,y_4,...,y_{n-2}, x_{i+1}\}.$$
By the proof of Claim \ref{lem-claim-3}, we know that the cycle $C_{ab}(i-1)=y_1y_2 \cdots y_{n-2}y_1$ can be composed as follows
 $C_{ab}(i-1)=x_{i+2}x_{i+3}\overrightarrow{C_n}x_{i-1}x_{i+2}$ or $C_{ab}(i-1)=x_{i-1} x_k \overleftarrow{C_n} x_{i+2} x_{k+1} \overrightarrow{C_n}x_{i-1}$ for some $k \in [n] \setminus \{i-2,i-1,i,i+1,i+2\}$. If $C_{ab}(i-1)=x_{i+2}x_{i+3}\overrightarrow{C_n}x_{i-1}x_{i+2}$, by Claim \ref{lem-claim-1} and Claim \ref{lem-claim-2}, since $C_{ab}(i-1)=y_1y_2 \cdots y_{n-2}y_1$ and $N_a(x_i)=\{y_1,y_3,...,y_{n-3}, x_{i+1}\}\mbox{ or } N_a(x_i)=\{y_2,y_4,...,y_{n-2}, x_{i+1}\}$, we can know that $N_a(x_i)=\{x_1,x_3,...,x_{n-3}, x_{n-1}\}$ when $i$ is even and $N_a(x_i)=\{x_2,x_4,...,x_{n-2}, x_n\}$ when $i$ is odd.
  If $C_{ab}(i-1)=x_{i-1} x_k \overleftarrow{C_n} x_{i+2} x_{k+1} \overrightarrow{C_n}x_{i-1}$ , by Claim \ref{lem-claim-1} and Claim \ref{lem-claim-2}, since $C_{ab}(i-1)=y_1y_2 \cdots y_{n-2}y_1$ and $N_a(x_i)=\{y_1,y_3,...,y_{n-3}, x_{i+1}\}$ or $N_a(x_i)=\{y_2,y_4,...,y_{n-2}, x_{i+1}\}$, we can also know that $k$ has the same parity as $i$, since otherwise the parity pattern of the vertices on $C_n$ would force $x_i$ to have neighbours in both parity classes, contradicting the previous description of $N_a(x_i)$. And $N_a(x_i)=\{x_1,x_3,...,x_{n-3}, x_{n-1}\}$ when $i$ is even and $N_a(x_i)=\{x_2,x_4,...,x_{n-2}, x_n\}$ when $i$ is odd.
 It remains to consider the vertex $x_n$. For this case, we apply Claim \ref{lem-claim-3} with $i=n-2$ and obtain a rainbow cycle $C_{ab}(n-2)=y_1y_2\cdots y_{n-2}y_1$. Repeating the above discussion, we can deduce that  $N_a(x_n)=\{x_1,x_3,...,x_{n-3}, x_{n-1}\}$. 
 The claim follows.
\end{proof}
  By Claim \ref{lem-claim-5}, we know that $G_1=\cdots =G_n=K_{\frac{n}{2}, \frac{n}{2}}$. Thus, the lemma follows. 
  \end{proof}

\section{The proof of Theorem \ref{[4,n]-vertex-pancyclicity}}
By Theorem \ref{Hamiltonian-cycle}, assume that $C_n = x_1 x_2 \cdots x_n x_1$ is a rainbow Hamiltonian cycle with $x_i x_{i+1} \in E(G_i)$ for each $i\in [n]$, where the subscripts are taken modulo $n$. Without loss of generality, we only need to show that there exists a rainbow cycle of length $\ell$ containing $x_1$ for every integer $\ell\in [4,n]$. By Theorem \ref{Hamiltonian-cycle}, the result follows when $\ell=n$. Next, we assume that $4\leq \ell \leq n-1$.

First we prove  Theorem \ref{[4,n]-vertex-pancyclicity} follows when $n \geq 8$ below.
\begin{claim}\label{x-(n-2)-cycle}
 There is a rainbow cycle of length $n-2$ containing $x_1$.
\begin{proof}
Suppose to the contrary that there is no rainbow cycle of length $n-2$ containing $x_1$. 
We assert that $x_1 x_3 \in E(G_1)$ or $x_2 x_4 \in E(G_3)$. Otherwise, we assume that $x_1 x_3 \notin E(G_1)$ and $x_2 x_4 \notin E(G_3)$. If $x_1 x_4 \in E(G_1) \cup E(G_2) \cup E(G_3)$, then $C_{n-2} = x_1 x_4 \cdots x_n x_1$ forms a rainbow cycle of length $n-2$ containing $x_1$, leading to a contradiction. Then $x_1 x_4 \notin E(G_1) \cup E(G_2) \cup E(G_3)$, it follows that $d_1(x_1) + d_3(x_4) \geq n.$
Define the sets:
$$
I_1 = \{ i \in [5, n] : x_1 x_i \in E(G_1) \} \mbox{ and }
I_3 = \{ i \in [4, n-1] : x_4 x_{i+1} \in E(G_3) \}.
$$
It follows from $x_2 x_4 \notin E(G_3)$ and $x_1 x_3 \notin E(G_1)$ that
$|I_1| \ge d_1(x_1) - 1$ and $|I_3| \ge d_3(x_4) - 1.$
Therefore,
$$
|I_1| + |I_3| \ge d_1(x_1) + d_3(x_4) - 2 \geq n-2.
$$
However, $I_1 \cup I_3 \subseteq [4, n]$ and $|I_1 \cup I_3|\leq n-3$, it means $I_1 \cap I_3 \neq \emptyset$. Choose an integer $j \in I_1 \cap I_3$, then $x_1 x_j \overleftarrow{C_n} x_4 x_{j+1} \overrightarrow{C_n}x_1$ forms a rainbow  cycle of length $n-2$ containing $x_1$, a contradiction.
Therefore, $x_1 x_3 \in E(G_1)$ or $x_2 x_4 \in E(G_3)$.

Then there exists a rainbow cycle $x_1 z x_4 \overrightarrow{C_n} x_n x_1$  of length $n-1$ containing $x_1$, where $z \in \{x_2, x_3\}$.
Similarly, we could get a rainbow cycle of length $n-2$ containing $x_1$ when $x_5x_8 \in E(G_5) \cup E(G_6) \cup E(G_7)$, a contradiction. Then $x_5x_8 \notin E(G_5) \cup E(G_6) \cup E(G_7)$.
Moreover, $x_5x_7 \notin E(G_5)$ and $x_6x_8 \notin E(G_7)$; otherwise, $x_1zx_4x_5x_7 \overrightarrow{C_n} x_n x_1$ or $x_1zx_4x_5x_6x_8 \overrightarrow{C_n} x_n x_1$ would be a rainbow cycle of length $n-2$,  a contradiction. Define
$$
I_5 = \{ i \in [n]\setminus \{5,6,7\}: x_5x_i \in E(G_5) \} \mbox{ and }
I_7 = \{ i \in  [n]\setminus \{5,6,7\}:x_8x_{i+1} \in E(G_7) \}.
$$
Then $I_5 \cup I_7 \subseteq [n]\setminus\{5,6,7\}$. 

We assert that $I_5 \cap I_7=\emptyset$. Otherwise, assume that $I_5 \cap I_7 \neq \emptyset$. Choose an integer $i \in I_5 \cap I_7$, if 
 $i \in [8,n]$, then $x_5x_i \overleftarrow{C_n} x_8 x_{i+1} \overrightarrow{C_n} x_5$ forms a rainbow cycle of length \(n-2\) containing $x_1$, a contradiction.
If $i \in [4]$, then $x_ix_5 \overleftarrow{C_n} x_{i+1}x_8 \overrightarrow{C_n} x_i$ is a rainbow cycle of length $n-2$ containing $x_1$, a contradiction.

It follows from $I_5 \cap I_7=\emptyset$ and $I_5 \cup I_7 \subseteq [n]\setminus\{5,6,7\}$ that 
$|I_5| + |I_7| = |I_5 \cup I_7| \le n - 3.$
Moreover, $|I_5| \ge d_5(x_5) - 1$ and $|I_7| \ge d_7(x_8) - 1$. Therefore,
$$
n - 3 \ge |I_5| + |I_7| \ge d_5(x_5) + d_7(x_8) - 2 \geq n-2,
$$
a contradiction. Hence, the claim holds.
\end{proof}
\end{claim}

We next prove the following key recursive claim.

\begin{claim}\label{fundamental-claim}
For each $\ell\in [8,n]$, if there is a rainbow cycle of length $\ell$ containing $x_1$, then there is a rainbow cycle of length $\ell-2$ containing $x_1$ in $\mathcal{G}$.
\end{claim}
\begin{proof}
 Suppose to the contrary that $x_1$ is contained in a rainbow cycle of length $\ell$ but not contained in a rainbow cycle of length $\ell-2$. By Claim \ref{x-(n-2)-cycle}, we have $\ell\leq n-1$. Without loss of generality, assume that $C_{\ell} = x_1 x_2 \cdots x_{\ell} x_1$ is a rainbow  cycle of length $\ell$ with $x_i x_{i+1} \in E(G_i)$ for each $i\in [\ell]$, where the subscripts are taken modulo $\ell$.

 We assert that there is a rainbow cycle $C_{\ell-1}=x_1 z x_4 \overrightarrow{C_{\ell}} x_1$ of length $\ell-1$ containing $x_1$, where $z=x_2$ or $z=x_3$ or $z\in V\setminus V(C_{\ell})$. If $x_1 x_3 \in E(G_1)$ or $x_2 x_4 \in E(G_3)$, then either $x_1 x_3\overrightarrow{C_{\ell}} x_1$ or $x_2 x_4 \overrightarrow{C_l}x_2$ is a rainbow $\ell-1$ cycle containing $x_1$. If $x_1 x_3 \notin E(G_1)$ and $x_2 x_4 \notin E(G_3)$, then  $x_1 x_4 \notin E(G_i)$ for each $i\in[3]\cup\{\ell+1,\ell+2,...,n\}$. Otherwise, if $x_1 x_4 \in E(G_i)$ for some $i\in[3]\cup\{\ell+1,\ell+2,...,n\}$, then $C_{\ell-2} = x_1 x_4 \cdots x_{\ell} x_1$ forms a rainbow cycle of length $\ell-2$ containing $x_1$, leading to a contradiction. Hence, $x_1 x_4 \notin E(G_i)$ for each $i\in[3]\cup\{\ell+1,\ell+2,...,n\}$, it follows that
$d_1(x_1)+d_3(x_4) \ge n$.
Define the following two sets:
$$
I_1 = \{ i \in [5, \ell] : x_1 x_i \in E(G_1) \} \mbox{ and }
I_3 = \{ i \in [4, \ell-1] : x_4 x_{i+1} \in E(G_3) \}.
$$
From $x_2 x_4 \notin E(G_3)$ and $x_1 x_3 \notin E(G_1)$, we derive:
$$
|I_1| \ge d_1(x_1, V(C_{\ell})) - 1 \mbox{ and }
|I_3| \ge d_3(x_4, V(C_{\ell})) - 1.
$$
Therefore,
$$
|I_1| + |I_3| \ge d_1(x_1, V(C_{\ell})) + d_3(x_4, V(C_{\ell})) - 2.
$$
Note that $I_1 \cup I_3 \subseteq [4, \ell]$. If $I_1 \cap I_3 \neq \emptyset$, then, choose an integer $j \in I_1 \cap I_3$, $x_1 x_j \overleftarrow{C_{\ell}} x_4 x_{j+1} \overrightarrow{C_{\ell}} x_1$ would be a rainbow  cycle of length $\ell-2$ containing $x_1$, which is a contradiction. Then $I_1 \cap I_3 = \emptyset$.
It follows that 
$|I_1| + |I_3| = |I_1 \cup I_3| \le \ell - 3.$
Combining with the previous inequality, we have:
$$
d_1(x_1, V(C_{\ell})) + d_3(x_4, V(C_{\ell})) \le \ell - 1.
$$
Since $d_1(x_1) + d_3(x_4) \ge n$, it follows that:
$$
d_1(x_1, V\setminus V(C_{\ell})) + d_3(x_4, V\setminus V(C_{\ell})) \ge n - \ell + 1.
$$
Consequently, there exists a vertex $y\in V\setminus V(C_{\ell})$ such that $x_1 y \in E(G_1)$ and $y x_4 \in E(G_3)$. Then,
$x_1 y x_4 \overrightarrow{C_{\ell}} x_1$ is a rainbow    cycle of length $\ell-1$ containing $x_1$.
From the above discussion, we obtain a rainbow cycle $C_{\ell-1}=x_1 z x_4 \overrightarrow{C_{\ell}} x_1$ of length $\ell-1$ containing $x_1$, where $z=x_2$ or $z=x_3$ or $z\in V\setminus V(C_{\ell})$.

If $x_5x_8 \in E(G_5) \cup E(G_6) \cup E(G_7)$, then $x_5x_8\overrightarrow{C_{\ell}}x_5$ is a rainbow cycle of length $\ell-2$ containing $x_1$, a contradiction. Then $x_5x_8 \notin E(G_5) \cup E(G_6) \cup E(G_7)$. 
Moreover, $x_5x_7 \notin E(G_n)$ and $x_6x_8 \notin E(G_7)$; otherwise,
either $x_5x_7\overrightarrow{C_{\ell-1}} x_5$ or $x_6x_8 \overrightarrow{C_{\ell-1}} x_6$ would be a rainbow cycle of length $\ell-2$ containing $x_1$, a contradiction.
Define
$$
I_n = \{ i \in [\ell] \setminus \{5,6,7\}: x_5x_i \in E(G_n) \} \mbox{ and }
I_7 = \{ i \in [\ell] \setminus \{5,6,7\}:  x_8x_{i+1} \in E(G_7) \}.
$$
Then $I_n \cup I_7 \subseteq [\ell] \setminus \{5,6,7\}$.
If $I_n \cap I_7 \neq \emptyset$, choose an integer $i \in I_n \cap I_7$, then $x_5x_i \overleftarrow{C_{\ell}} x_8 x_{i+1} \overrightarrow{C_{\ell}} x_{5}$ is a rainbow cycle of length $\ell-2$ containing $x_1$, contradiction. 
Thus, $I_n \cap I_7 = \emptyset$, it follows that
$$
|I_n| + |I_7| = |I_n \cup I_7| \le \ell - 3.
$$
Moreover, $|I_n| \ge d_n(x_5, V(C_{\ell})) - 1$ and $|I_7| \ge d_7(x_8, V(C_{\ell})) - 1$. Therefore,
$$
\ell - 3 \ge |I_n| + |I_7| \ge d_n(x_5, V(C_{\ell})) + d_7(x_8, V(C_{\ell})) - 2,
$$
so
$$
d_n(x_5, V(C_{\ell})) + d_7(x_8, V(C_{\ell})) \le \ell - 1.
$$
Consequently,
$$
d_n(x_5, V\setminus V(C_{\ell})) + d_7(x_8, V\setminus V(C_{\ell})) \ge n - \ell + 1.
$$
\begin{figure}
    \centering
    \includegraphics[width=1\linewidth]{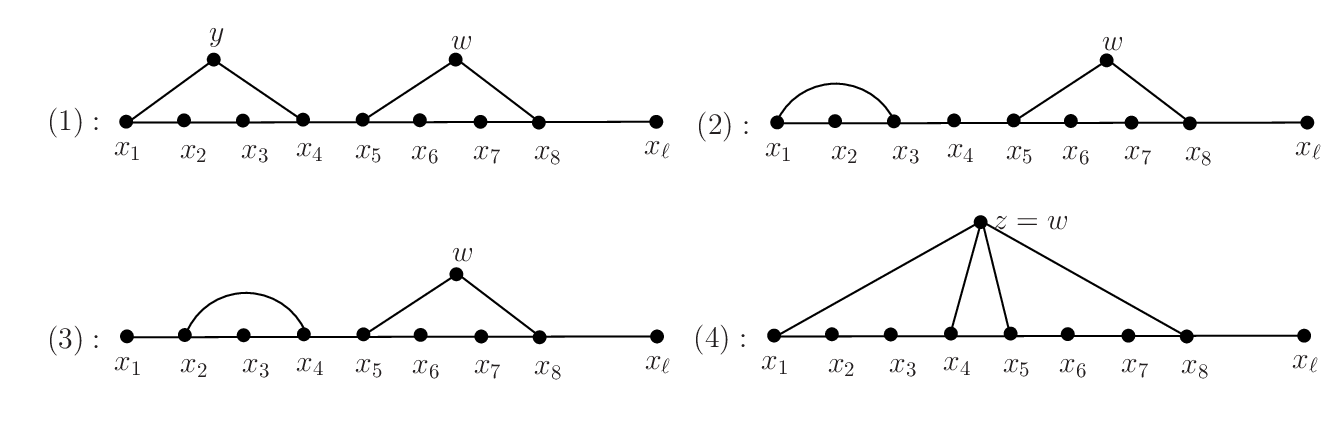}
    \caption{The rainbow cycle of length $\ell-2$}
    \label{fig-3}
\end{figure}
Since $|V\setminus V(C_{\ell})|=n-\ell$, there is at least one vertex $w\in V\setminus V(C_{\ell})$ such that 
$x_5w \in E(G_{n})$, $wx_8 \in E(G_7)$.
If $w=z$, then $x_1zx_5\overrightarrow{C_{\ell}}x_1$ is a rainbow cycle of length $\ell-2$ containing $x_1$, see (4) of Figure \ref{fig-3}, a contradiction. If $w\neq z$, then $x_1zx_4x_5wx_8\overrightarrow{C_{\ell}}x_1$ is a rainbow cycle of length $\ell-2$ containing $x_1$, where $x_1z\in E(G_1)$, $zx_4\in E(G_3)$, $x_5w\in E(G_n)$ and $wx_8\in E(G_7)$. See (1), (2) and (3) of Figure \ref{fig-3}. This contradicts the hypothesis that there is not a rainbow cycle of $\ell-2$ containing $x_1$, the claim follows.
\end{proof}

By Theorem \ref{Hamiltonian-cycle}, $x_1$ is contained in a rainbow Hamiltonian cycle. By Lemma \ref{[n-1]-vertex-pancyclicity}, we know that $G_1=\cdots =G_n=K_{\frac{n}{2},\frac{n}{2}}$ or each vertex is contained in a rainbow cycle of length $n-1$. If the former holds, then the result follows. If the latter holds, then $x_1$ is contained in a rainbow cycle of length $n-1$. Using Claim \ref{fundamental-claim}, then $x_1$ is contained in a rainbow cycle of length $\ell$ for each $\ell\in [6,n]$. Now we prove the following claim:
\begin{claim}\label{fundamental-claim-7-5}
There is a rainbow cycle    of length $5$ containing $x_1$ in $\mathcal{G}$ or  $G_1=\cdots =G_n=K_{\frac{n}{2},\frac{n}{2}}$.
\end{claim}
\begin{proof}
Suppose to the contrary that $G_1=\cdots =G_n=K_{\frac{n}{2},\frac{n}{2}}$ does not hold and there is no rainbow cycle of length $5$ containing $x_1$. 
Then by Lemma \ref{[n-1]-vertex-pancyclicity}, without loss of generality, assume that $C_7 = x_1x_2\cdots x_7x_1$ is a rainbow cycle of length $7$ containing $x_1$. 
 
 We assert that there is a rainbow cycle of length $6$ containing $x_1$. If $x_1 x_3 \in E(G_1)$ or $x_2 x_4 \in E(G_8)$, then either $x_1x_3x_4x_5x_6x_7x_1$ or $x_1x_2 x_4x_5x_6x_7x_1$ is a rainbow cycle of length $6$ containing $x_1$. Then $x_1 x_3 \notin E(G_1)$ and $x_2x_4 \notin E(G_8)$.
If $x_1 x_4 \in E(G_i)$ for some $i\in[n]\setminus \{4,5,6,7\}$, then $x_1 x_4 x_5x_6x_7 x_1$ forms a rainbow cycle of length $5$ containing $x_1$, leading to a contradiction. Hence, $x_1 x_4 \notin E(G_i)$ for each $i\in[n]\setminus \{4,5,6,7\}$, it follows that $d_1(x_1)+d_8(x_4) \ge n$.
Define the following two sets:
$$
I_1 = \{ i \in \{5,6,7\} : x_1 x_i \in E(G_1) \} \mbox{ and }
I_8 = \{ i \in \{4,5,6\} : x_4 x_{i+1} \in E(G_8) \}.
$$
From $x_2 x_4 \notin E(G_8)$, $x_1 x_3 \notin E(G_1)$ and $x_1 x_4 \notin E(G_i)$ for each $i\in[n]\setminus \{4,5,6,7\}$, we derive:
$$
|I_1| \ge d_1(x_1, V(C_7)) - 1 \mbox{ and }
|I_8| \ge d_8(x_4, V(C_7)) - 1.
$$
Therefore,
$$
|I_1| + |I_8| \ge d_1(x_1, V(C_7)) + d_8(x_4, V(C_7)) - 2.
$$
Note that $I_1 \cup I_8 \subseteq \{4,5,6,7\}$. If $I_1 \cap I_8 \neq \emptyset$, then choose an integer $i \in I_1 \cap I_8$, $x_1 x_i \overleftarrow{C_7} x_4 x_{i+1} \overrightarrow{C_7} x_1$ would be a rainbow  cycle of length $5$ containing $x_1$, which is a contradiction. Then $I_1 \cap I_8 = \emptyset$.
It follows that 
$|I_1| + |I_8| = |I_1 \cup I_8| \le 4.$
Combining with the previous inequality, we have:
$$
d_1(x_1, V(C_7)) + d_8(x_4, V(C_7)) \le 6.
$$
Since $d_1(x_1) + d_8(x_4) \ge n$, it follows that:
$$
d_1(x_1, V\setminus V(C_7)) + d_8(x_4, V\setminus V(C_7)) \ge n -6.
$$
Consequently, there exists a vertex $y\in V\setminus V(C_7)$ such that $x_1 y \in E(G_1)$ and $y x_4 \in E(G_8)$. Then,
$x_1 y x_4x_5x_6x_7x_1$ is a rainbow cycle of length $6$ containing $x_1$.

From the above discussion, we could obtain a rainbow cycle $C_6=x_1 z x_4x_5x_6x_7x_1$ of length $6$ containing $x_1$, where $x_1z\in E(G_1)$, $zx_4\in E(G_8)$ and $z\in \{x_2,x_3\}\cup V\setminus V(C_7)$. If $x_5 x_7 \in E(G_5)$ or $x_6 x_1 \in E(G_7)$, then either $x_1zx_4x_5x_7x_1$ or $x_1zx_4x_5x_6x_1$ is a rainbow cycle of length $5$ containing $x_1$. Then $x_5 x_7 \notin E(G_5)$ and $x_6x_1 \notin E(G_7)$. By a similar argument applied to $x_5$ and $x_1$, we could find a vertex $w$ such that $x_5w\in E(G_5)$, $wx_1\in E(G_7)$ and $w\in V\setminus V(C_7)$.
If $z \neq w$, then $x_1zx_4x_5wx_1$ is a rainbow cycle of length $5$ containing $x_1$. If $z = w$, then $x_1x_2x_3x_4zx_1$ is a rainbow cycle of length $5$ containing $x_1$.
Thus, the claim is proved.
\end{proof}

\begin{claim}\label{4-cycle}

There is a rainbow cycle of length 4 containing $x_1$.

\end{claim}

\begin{proof}

By Claims \ref{x-(n-2)-cycle}, \ref{fundamental-claim}, \ref{fundamental-claim-7-5} and Lemma \ref{[n-1]-vertex-pancyclicity}, if $G_1=\cdots =G_n=K_{\frac{n}{2},\frac{n}{2}}$, then there is  a rainbow cycle of length 4 containing $x_1$. Otherwise, there exists a rainbow cycle $C_5 = x_1x_2x_3x_4x_5x_1$ of length 5.
Without loss of generality, assume that $x_i x_{i+1} \in E(G_i)$ for each $i$, where $x_6=x_1$.
Consider the edge $x_1x_3$.
 If $x_1x_3 \in E(G_1)$ or $x_1x_3 \in E(G_2)$, then $x_1x_3x_4x_5x_1$ forms a rainbow $C_4$ containing $x_1$.
If $x_1x_3 \notin E(G_1) \cup E(G_2)$, then we have $d_6(x_1) + d_7(x_3) \geq n$. Moreover, note that $x_1x_3, x_1x_4 \notin E(G_6)$ and $x_3x_5, x_3x_1 \notin E(G_7)$. 
    Consequently,
    $d_6(x_1, V\setminus V(C_5)) + d_7(x_3, V\setminus V(C_5)) \geq n - 4$. 
    Since $|V\setminus V(C_5)| = n-5$, there exists a vertex $s \in V\setminus V(C_5)$ such that $x_1s \in E(G_6)$ and $sx_3 \in E(G_7)$. 
    Then $x_1x_2x_3sx_1$ is a rainbow cycle of length 4 containing $x_1$.
Thus, in all cases, we obtain a rainbow cycle of length 4 containing $x_1$.
\end{proof}
Then Theorem \ref{[4,n]-vertex-pancyclicity} follows when $n \geq 8$. Now we discuss the case $4 \leq n \leq 7$. 
\begin{claim}\label{small n}
The result follows when $4 \leq n\leq 7$.
\end{claim}
\begin{proof}   
It is clear that the claim follows when $n = 4, 5$. 
If  $n = 6$, then Lemma \ref{[n-1]-vertex-pancyclicity} implies that there is a rainbow cycle of length $5$ containing $x_1$ in $\mathcal{G}$ or $G_1=\cdots =G_6=K_{3, 3}$. If there is a rainbow cycle of length $5$ containing $x_1$, assume that $C_6=x_1x_2\cdots x_6x_1$ is a rainbow Hamiltonian cycle with $x_i x_{i+1} \in E(G_i)$ for $i \in [6]$, where $x_7=x_1$. Suppose to the contrary that there is no rainbow cycle of length $4$ containing $x_1$.  If $c(x_2x_6) \cap \{2,3,4,5\} \leq 2$, then $d_a(x_2) + d_b(x_6) \geq 6$ for $a,b \notin c(x_2x_6)\cap \{2,3,4,5\}$, then $d_a(x_2, \{x_3,x_4,x_5\}) + d_b(x_6,\{x_3,x_4,x_5\}) \geq 4$. There is a rainbow cycle $x_1x_2x_3x_6x_1$ or $x_1x_2x_4x_6x_1$ or  $x_1x_2x_5x_6x_1$ of length $4$ containing $x_1$. It follows that $c(x_2x_6)  \cap \{2,3,4,5\} \geq 3 $. By a similar argument, we obtain $c(x_1x_3)  \cap \{3,4,5,6\} \geq 3 $. Thus, there is a rainbow cycle $x_1x_3x_2x_6x_1$ of length $4$ containing $x_1$, a contradiction. So Theorem \ref{[4,n]-vertex-pancyclicity} holds when $n = 6$. 

If  $n = 7$, assume that $C_7=x_1x_2\cdots x_7x_1$ is a rainbow Hamiltonian cycle with $x_i x_{i+1} \in E(G_i)$ for $i \in [7]$, where $x_8=x_1$. By Lemma \ref{[n-1]-vertex-pancyclicity}  we know there is a rainbow cycle of length $6$ containing $x_1$ in $\mathcal{G}$. Suppose to the contrary that there is no rainbow cycle of length $5$ containing $x_1$. We assert that $x_1 x_3 \in E(G_1)$ or $x_2x_4 \in E(G_3)$. Otherwise, $x_1 x_3 \notin E(G_1)$ and $x_2x_4 \notin E(G_3)$.
If $x_1 x_4 \in E(G_i)$ for some $i\in\{1,2,3\}$, then $x_1 x_4 x_5x_6x_7 x_1$ forms a rainbow cycle of length $5$ containing $x_1$, leading to a contradiction. Hence, $x_1 x_4 \notin E(G_i)$ for each $i\in \{1,2,3\}$, it follows that $d_1(x_1)+d_3(x_4) \ge 7$.
Define the following two sets:
$$
I_1 = \{ i \in \{5,6,7\} : x_1 x_i \in E(G_1) \} \mbox{ and }
I_3 = \{ i \in \{4,5,6\} : x_4 x_{i+1} \in E(G_3) \}.
$$
From $x_2 x_4 \notin E(G_3)$, $x_1 x_3 \notin E(G_1)$ and $x_1 x_4 \notin E(G_i)$ for each $i\in \{1,2,3\}$, we derive:
$$
|I_1| \ge d_1(x_1) - 1 \mbox{ and }
|I_3| \ge d_3(x_4) - 1.
$$
Therefore,
$$
|I_1| + |I_3| \ge d_1(x_1) + d_3(x_4) - 2.
$$
Note that $I_1 \cup I_3 \subseteq \{4,5,6,7\}$. If $I_1 \cap I_3 \neq \emptyset$, then, choose an integer $i \in I_1 \cap I_3$, $x_1 x_i \overleftarrow{C_7} x_4 x_{i+1} \overrightarrow{C_7} x_1$ would be a rainbow  cycle of length $5$ containing $x_1$, which is a contradiction. Then $I_1 \cap I_3 = \emptyset$.
It follows that 
$|I_1| + |I_3| = |I_1 \cup I_3| \le 4.$
Combining with the previous inequality, we have:
$$
d_1(x_1) + d_3(x_4) \le 6.
$$
Since $d_1(x_1) + d_3(x_4) \ge 7$, a contradiction. So $x_1 x_3 \in E(G_1)$ or $x_2x_4 \in E(G_3)$. By a similar argument applied to $x_5$ and $x_1$, we can obtain that $x_5 x_7 \in E(G_5)$ or $x_6x_1 \in E(G_7)$. Then $x_1yx_4x_5zx_1$ is a rainbow cycle of length $5$ where $y \in \{x_2,x_3\}$, $z \in \{x_6,x_7\}$ and $x_1y\in E(G_1)$, $yx_4\in E(G_3)$, $x_5z\in E(G_5)$, $zx_1\in E(G_7)$.
Since there exists a rainbow cycle $x_1yx_4x_5zx_1$ of length 5, by a similar  discussion with the case in Claim \ref{4-cycle} we obtain a rainbow cycle of length 4 containing $x_1$. Then the claim follows when $n=7$.


\end{proof}
Combining the above discussion, the proof of Theorem \ref{[4,n]-vertex-pancyclicity} is completed. $\hfill\square$

\section{The proof of Theorem \ref{vertex-pancyclicity-ore}}
If $n=3$, then $G_1=G_2=G_3=K_3$. The result follows.
If $n\geq 4$, then $\frac{4n-3}{3}>n$.
By Theorem \ref{[4,n]-vertex-pancyclicity}, we know that each vertex is contained in a rainbow cycle of length $\ell$ for each $\ell\in[4,n]$. 
For each vertex $x\in V$, if $xy\in E(G_i)$ for some $i\in [n]$ and $xy$ is not a strong edge, then $d_j(x)+d_k(y)\geq \frac{4n-3}{3}$ for any two integers $j,k\in[n]$. It is easy to find a vertex $z\in N_{j}(x)\cap N_{k}(y)$. Then $xyzx$ is a rainbow triangle. The result follows. 
It follows that $xy$ is a strong edge if $xy\in E(G_i)$ for some $i\in [n]$. 

Note that $d_i(x)\geq \frac{4n-3}{3}-(n-1)\geq \frac{n}{3}$ for each $i\in [n]$. 
It follows from $n\geq 4$ that $d_i(x)\geq 2$. Then $N_{G_i}(x)$ is an independent set in $G_i$ for each $i\in [n]$. Choose two vertices $x_1,x_2\in N_{i}(x)$, then $$\frac{4n-3}{3}\leq d_{i}(x_1)+d_{j}(x_2)\leq 2(n-d_i(x)-1)+2.$$
Then we have $d_i(x)\leq \frac{n}{3}-\frac{1}{2}$.
Choose a vertex $y\in V\setminus (N_{G_i}(x)\cup \{x\})$, then $xy\notin E(G_i)$ for each $i\in [n]$. 
Then 
$\frac{4n}{3}-1\leq d_{i}(x)+d_{j}(y)\leq d_{i}(x)+n-1\leq \frac{4n}{3}-\frac{3}{2},$ a contradiction. The result follows.
$\hfill\square$

\section{Acknowledgments}

            This work was supported by the National Key R\&D Program of China (No. 2023YFA1009602) and the National Natural Science Foundation of China (Grant No. 12231018). The authors declare that they have no conflict of interest. No data were used for the research described in this article.

\bibliographystyle{abbrv}
\bibliography{main}





\end{document}